\documentclass[12pt]{article}

\usepackage{fullpage}
\usepackage{amssymb}
\usepackage{amsmath}
\usepackage{amsthm}
\usepackage{hyperref}
\usepackage{graphicx}
\newtheorem{theorem}{Theorem}
\newtheorem{corollary}{Corollary}

\newtheorem{lemma}{Lemma}
\newtheorem{conjecture}{Conjecture}

\begin{document}

\title{Counting the number of isosceles triangles in rectangular regular grids}
\date{February 6, 2017}
\author{Chai Wah Wu\\ IBM T. J. Watson Research Center\\ P. O. Box 218, Yorktown Heights, New York 10598, USA\\e-mail: chaiwahwu@member.ams.org}
\maketitle

\begin{abstract}
In general graph theory, the only relationship between vertices are expressed via the edges. When the vertices are embedded in an Euclidean space, the geometric relationships between vertices and edges can be interesting objects of study.
We look at the number of isosceles triangles where the vertices are points on a regular grid and show that they satisfy a recurrence relation when the grid is large enough.  We also derive recurrence relations for the number of acute, obtuse and right isosceles triangles.
\end{abstract}

\section{Introduction}\label{sec:intro}
In general graph theory, the relationship between vertices are expressed via the edges of the graph.  In geometric graph theory, the vertices and edges have geometric attributes that are important as well.  For instance, a random geometric graph is constructed by embedding the vertices randomly in some metric space and connecting two vertices by an edge if and only if they are close enough to each other.  When the vertices lie in an Euclidean space, the edges of vertices can form geometric objects such as polygons.  In \cite{nara:polygon:2003}, the occurrence of polygons is studied. In \cite{bautista:triangles:2013} the number of nontrivial triangles is studied.  In this note, we consider this problem when the vertices are arranged on a regular grid.  The study of the abundance (or sparsity) of such subgraphs or network motifs \cite{milo:motifs:2002} is important in the characterization of complex networks.

Consider an $n$ by $k$ rectangular regular grid $G$ with $n,k\geq 2$.  A physical manifestation of this pattern, called geoboard, is useful in teaching elementary geometric concepts \cite{geoboard}. Let $3$ distinct points be chosen on the grid such that they form the vertices of a triangle with nonzero area (i.e. the points are not collinear).  
In OEIS sequences \href{http://oeis.org/A271910}{A271910}, \href{http://oeis.org/A271911}{A271911}, \href{http://oeis.org/A271912}{A271912}, \href{http://oeis.org/A271913}{A271913}, \href{http://oeis.org/A271915}{A271915} \cite{oeis}, the number of such triangles that are isosceles are listed for various $k$ and $n$.  Neil Sloane made the conjecture that for a fixed $n\geq 2$, the number of isosceles triangles in an $n$ by $k$ grid, denoted as $a_n(k)$, satisfies the recurrence relation $a_n(k) = 2a_n(k-1)-2a_n(k-3)+a_n(k-4)$ for $k > K(n)$ for some number $K(n)$.  The purpose of this note is to show that this conjecture is true and give an explicit form of $K(n)$.  In particular, we show that $K(n) = (n-1)^2+3$ if $n$ is even and $K(n) = (n-1)^2+2$ if $n$ is odd and that this is the best possible value for $K(n)$.

\section{Counting isosceles triangles} \label{sec:count}
We first start with some simple results:

\begin{lemma} \label{lem:one}
If $x$,$y$,$u$,$w$ are integers such that
$0 < x$,$u \leq n$ and $y > \frac{n^2}{2}$, then $x^2+y^2=u^2+w^2$ implies that $x = u$ and $y = w$.
\end{lemma}
\begin{proof} If $y\neq w$, then
$|y^2-w^2| \geq 2y-1 > n^2-1$.  
On the other hand, $|y^2-w^2| = |u^2-x^2| < n^2$, a contradiction.
\end{proof}

\begin{lemma} \label{lem:two}
If $x$,$y$,$u$,$w$ are integers such that
$0\leq x$,$u \leq n$ and $y > \frac{n^2+1}{2}$, then $x^2+y^2=u^2+w^2$ implies that $x = u$ and $y = w$.
\end{lemma}
\begin{proof} If $y\neq w$, then
$|y^2-w^2| \geq 2y-1 > n^2$.  
On the other hand, $|y^2-w^2| = |u^2-x^2| \leq n^2$, a contradiction.
\end{proof}

\begin{lemma} \label{lem:three}
\[ 2\sum_{m=1}^{\lfloor\frac{n-1}{2}\rfloor} n-2m = \left\lfloor \frac{(n-1)^2}{2}\right\rfloor\]  
\end{lemma}
\begin{proof}
If $n$ is odd, then 
\[ 2\sum_{m=1}^{\lfloor\frac{n-1}{2}\rfloor} n-2m = 2\sum_{m=1}^{\frac{n-1}{2}} n-2m = 
n(n-1) - 2\frac{n-1}{2}\frac{n+1}{2} = \frac{(n-1)^2}{2} \]
If $n$ is even, then
\[ 2\sum_{m=1}^{\lfloor\frac{n-1}{2}\rfloor} n-2m = 2\sum_{m=1}^{\frac{n-2}{2}} n-2m = 
n(n-2) -  2\frac{n-2}{2}\frac{n}{2}  = \frac{(n-1)^2-1}{2} = \left\lfloor \frac{(n-1)^2}{2}\right\rfloor \]
\end{proof}

Our main result is the following:
\begin{theorem} \label{thm:main}
Let $a_n(k)$ be the number of isosceles triangles of nonzero area formed by 3 distinct points in an $n$ by $k$ grid.  Then
$a_n(k) = 2a_n(k-1)-2a_n(k-3)+a_n(k-4)$ for $k > (n-1)^2+3$.
\end{theorem}

\begin{proof}
for $k > 2$, let the $n$ by $k$ array be decomposed into $3$ parts consisting of the first column, the middle part (of size $n$ by $k-2$) and the last column denoted as $p_1$, $p_2$ and $p_3$ respectively.

Let $b_n(k)$ be the number of isosceles triangles in the $n$ by $k$ array with vertices in the last column $p_3$ and
let $c_n(k)$ be the number of isosceles triangles in the $n$ by $k$ array with vertices in the first and last columns $p_1$ and $p_3$.
It is clear that $b_n(k) = a_n(k)-a_n(k-1)$.
Furthermore, $b_n(k) = b_n(k-1) + c_n(k)$. 
Let us partition the isosceles triangles corresponding to $c_n(k)$ into $2$ groups, $A(k)$ and $B(k)$ where $A(k)$ are triangles with all $3$ vertices in $p_1$ or $p_3$ and $B(k)$ are triangles with a vertex in each of $p_1$, $p_2$ and $p_3$. 
Since $k > n$, all triangles in $A(k)$ must be of the 
form where the two vertices in $p_1$ (resp. $p_3$) are an even number of rows apart and the third vertex is in $p_3$ (resp. $p_1$)
in the middle row between them.  Since $k > (n-1)^2+1$, these triangles are all acute (we'll revisit this later).  Let us count how many such triangles there are.  There are $n-2m$ pairs of vertices which are $2m$ rows apart, for $1\leq m\leq \lfloor\frac{n-1}{2}\rfloor$.  Thus the total number of triangles in $A(k)$ is
\[ 2\sum_{m=1}^{\lfloor\frac{n-1}{2}\rfloor} n-2m = \left\lfloor \frac{(n-1)^2}{2}\right\rfloor\]  
by Lemma \ref{lem:three}.
Next we consider the isosceles triangles in $B(k)$.
Let $e_1$ be the edge between the vertex in $p_1$ and the vertex in $p_2$ and $e_2$ be the edge between the vertex in $p_2$ and the vertex in $p_3$ and $e_3$ be the edge between the vertices in $p_1$ and $p_3$.  There are 2 cases.  In case 1, the length of $e_1$ is equal to the length of $e_2$ and is expressed as $x^2+y^2=u^2+w^2$ with $0\leq x, u \leq n-1$ and $y+w=k-1$.  Without loss of generality, we pick $y$ to be the larger of $y$ and $w$, i.e., $y \geq \frac{k}{2}$.
This is illustrated in Fig. \ref{fig:triangle}.

\begin{figure}[htbp]
\centerline{\includegraphics{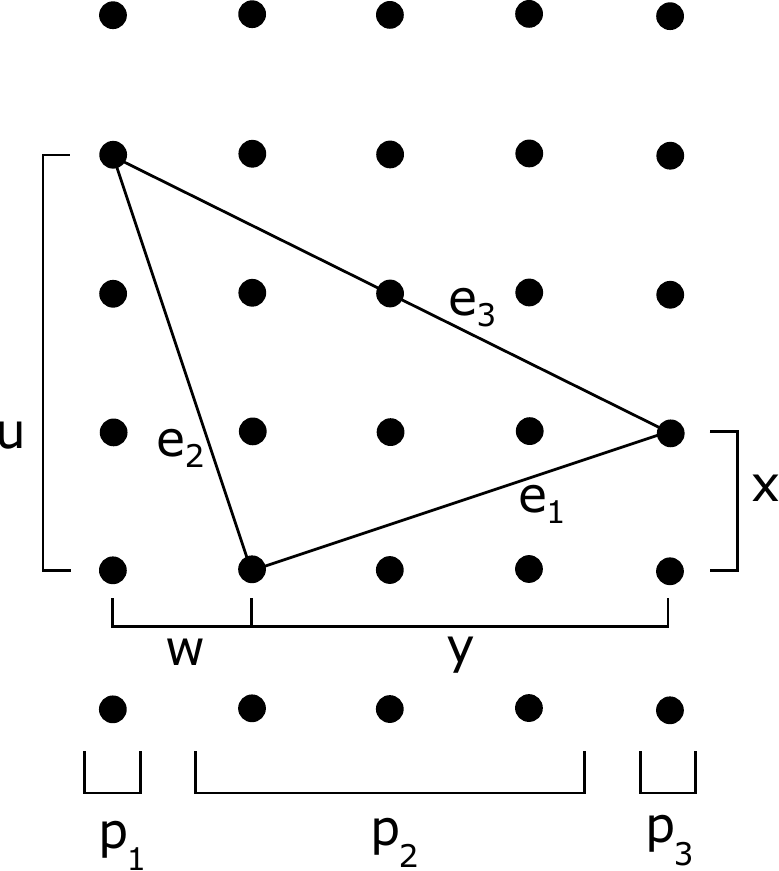}}
\caption{Illustrating a triangle in $B(k)$.}\label{fig:triangle}
\end{figure}

Since $k > (n-1)^2+1$, Lemma \ref{lem:two} implies that $x = u$ and $y = w$ and this is only possible if $k$ is odd and the vertex in $p_1$ and in $p_3$ must be at the same row and $p_2$ in a different row.  Note that such triangles are right triangles for $n=2$ and obtuse for $n > 2$.  In case 2, the length of $e_3$ is equal to the length of either $e_1$ or $e_2$.  Lemma \ref{lem:two} shows that in this case $y = k-1$ which is not possible.  This implies that $c_n(k) = P(n,2) + \lfloor \frac{(n-1)^2}{2}\rfloor= n(n-1) +  \lfloor \frac{(n-1)^2}{2}\rfloor$ if $k$ is odd and $c_n(k) =  \lfloor \frac{(n-1)^2}{2}\rfloor$ if $k$ is even.
Thus we have $a_n(k) = a_n(k-1)+b_n(k) = a_n(k-1)+b_n(k-1)+c_n(k) = a_n(k-1)+a_n(k-1)-a_n(k-2)+c_n(k) = 2a_n(k-1)-a_n(k-2)+c_n(k)$.
Since $k-2 > n$ and $k-2 > (n-1)^2 +1$, we can apply the same analysis to find that $c_n(k-2) = c_n(k)$.
Since $a_n(k-2) = 2a_n(k-3)-a_n(k-4)+c_n(k-2)$, we get $a_n(k) = 2a_n(k-1)-2a_n(k-3)+a_n(k-4)-c_n(k-2)+c_n(k) = 2a_n(k-1)-2a_n(k-3)+a_n(k-4)$.
\end{proof}

It is clear that $a_n(k) = a_k(n)$.
The above argument also shows that:
\begin{theorem} \label{thm:main2}
Suppose that $ k > (n-1)^2+1$.  Then
$a_n(k) = 2a_n(k-1)-a_n(k-2)+n(n-1) + \lfloor \frac{(n-1)^2}{2}\rfloor$ if $k$ is odd and $a_n(k) = 2a_n(k-1)-a_n(k-2) + \lfloor \frac{(n-1)^2}{2}\rfloor$ if $k$ is even.  
\end{theorem}

\section{Obtuse, acute and right isosceles}
As noted above, the triangles in $A(k)$ are acute and the triangles in $B(k)$ are right for $n=2$ and obtuse for $n>2$.
This implies that Theorem \ref{thm:main} is also true when restricted to the set of acute isosceles triangles and restricted obtuse or right isosceles triangles if $n>2$.
As for Theorem \ref{thm:main2}, we have the following similar results:

\begin{theorem} \label{thm:main2-acute}
Let $a^a_n(k)$ be the number of acute isosceles triangles of nonzero area formed by 3 distinct points in an $n$ by $k$ grid.
Then
$a^a_n(k) = 2a^a_n(k-1)-a^a_n(k-2) + \lfloor \frac{(n-1)^2}{2}\rfloor$ for $ k > (n-1)^2+1$.
\end{theorem}

\begin{theorem} \label{thm:main2-obtuse}
Let $a^o_n(k)$ be the number of obtuse isosceles triangles of nonzero area formed by 3 distinct points in an $n$ by $k$ grid. Suppose $ k > \max(3,(n-1)^2+1)$.
Then $a^o_n(k) = 2a^o_n(k-1)-a^o_n(k-2)+n(n-1) $ if $k$ is odd and $a^o_n(k) = 2a^o_n(k-1)-a^o_n(k-2) $ if $k$ is even.  
\end{theorem}

\begin{theorem} \label{thm:main2-rectangular}
Let $a^r_n(k)$ be the number of right isosceles triangles of nonzero area formed by 3 distinct points in an $n$ by $k$ grid.
Then
 $a^r_n(k) = 2a^r_n(k-1)-a^r_n(k-2)$ for $ k > \max(3,(n-1)^2+1)$.
\end{theorem}

\section{Pythagorean triples and a small improvement}
In Theorem \ref{thm:main} we have given an explicit form for the bound $K(n)$ in the conjecture described in Section \ref{sec:intro}.  We next show that for odd $n$, this bound can be reduced by $1$.
Let us consider the case $k = (n-1)^2+1$.
Consider the triangles in $B(k)$.  
Again case 2 is impossible.
For case 1, the length of $e_1$ is equal to the length of $e_2$ and is expressed as $x^2+y^2=u^2+w^2$ with $0\leq x, u \leq n-1$ and  $y\geq \frac{k}{2}$ and $y\geq w$. 
If $x > 0$, this implies that $u > 0$ and we can apply Lemma \ref{lem:one} to show that the only isosceles triangles in $B(k)$ are as in Theorem \ref{thm:main}.
Suppose $x = 0$.  We can eliminate $y=w$ since this results in $u=0$ and a collinear set of vertices.  For $n=2$, $k=2$, it is clear that $w=0$ and $y=w+1 = 1$.  For $n>2$, since $u\leq n-1$, this again means that $y = w+1$ as otherwise $y^2-w^2 \geq 4y-4 \geq 2k-4 > (n-1)^2$.  This implies that $k$ must necessarily be even  such that $a$,$b$,$c$ are integers forming a Pythagorean triple satisfying $a^2 = b^2+c^2$ where 
$a = \frac{k}{2}$, $b=a-1$ and $c = \sqrt{a^2-b^2} = \sqrt{k-1}$.  
If $n$ is odd, then $k=(n-1)^2+1$ is odd, and the case of $x=0$ cannot occur.  Thus
Theorem \ref{thm:main} can be improved to:
 
\begin{theorem} \label{thm:main-even}
Suppose $n$ is odd.  Then
$a_n(k) = 2a_n(k-1)-2a_n(k-3)+a_n(k-4)$ for $k > (n-1)^2+2$.
\end{theorem}

We can rewrite this in an inhomogeneous form of lower degree:

\begin{theorem} \label{thm:main2-even}
Suppose that $n $ is odd and  $ k > (n-1)^2$. Then
$a_n(k) = 2a_n(k-1)-a_n(k-2)+n(n-1) + \lfloor \frac{(n-1)^2}{2}\rfloor$ if $k$ is odd and $a_n(k) = 2a_n(k-1)-a_n(k-2) + \lfloor \frac{(n-1)^2}{2}\rfloor$ if $k$ is even.  
\end{theorem}

Again, Theorem \ref{thm:main-even} is still valid when restricted to acute isosceles triangles.  When $n=3$ and $k=(n-1)^2+1=5$, two of the triangles in $B(k)$ are right isosceles and for $n>3$ and $k=(n-1)^2+1$, all triangles in $B(k)$ are obtuse triangles.  Thus we have

\begin{theorem} \label{thm:main-even-aor}
Suppose $n$ is odd.   Then
$a^a_n(k) = 2a^a_n(k-1)-2a^a_n(k-3)+a^a_n(k-4)$ for $k > (n-1)^2+2$.  Furthermore,
$a^o_n(k) = 2a^o_n(k-1)-2a^o_n(k-3)+a^o_n(k-4)$ and 
$a^r_n(k) = 2a^r_n(k-1)-2a^r_n(k-3)+a^r_n(k-4)$ for $k > \max(7,(n-1)^2+2)$.
\end{theorem}
 
Similiarly we have:

\begin{theorem} \label{thm:main2-even-obtuse}
Suppose $n$ is odd and $ k > \max(5,(n-1)^2)$. Then
$a^o_n(k) = 2a^o_n(k-1)-a^o_n(k-2)+n(n-1) $ for $k$ odd and $a^o_n(k) = 2a^o_n(k-1)-a^o_n(k-2) $ for $k$ even.  
\end{theorem}

When $n$ is even and $k=(n-1)^2+1$, the above analysis shows that there are 4 extra isosceles triangles in $B(k)$ due to the case $x=0$.  These triangles are right triangles for $n=2$ and obtuse for $n> 2$. This is illustrated in Fig. \ref{fig:two} for the case $k=10$, $n=4$. This means that when restricted to acute isosceles triangles (or right isosceles triangles with $n>2$), the condition that $n$ is odd is not necessary in Theorem \ref{thm:main-even}.  In addition Theorems \ref{thm:main2-acute} and \ref{thm:main2-rectangular} can be improved to:

\begin{figure}[htbp]
\centerline{\includegraphics{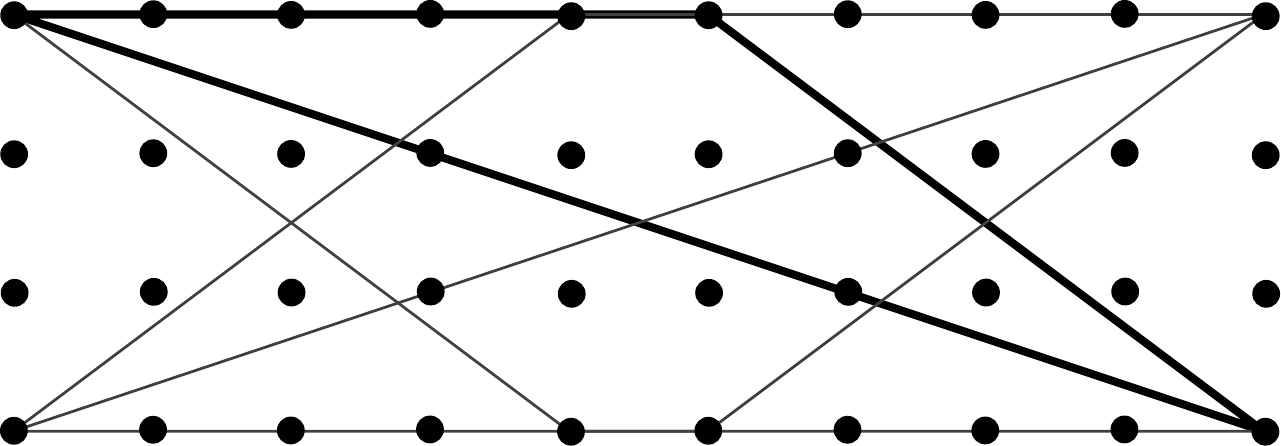}}
\caption{$4$ obtuse isosceles triangles in $B(k)$ corresponding to $x=0$ for the case $k=10$, $n=4$.}\label{fig:two}
\end{figure}

\begin{theorem} \label{thm:main2-even-acute}
Suppose $ k > (n-1)^2$. Then
$a^a_n(k) = 2a^a_n(k-1)-a^a_n(k-2)+\left\lfloor \frac{(n-1)^2}{2}\right\rfloor $.
\end{theorem}

This can be rewritten in homogeneous form as:

\begin{corollary} \label{cor:main2-even-acute}
Suppose $ k > (n-1)^2+1$. Then $a^a_n(k) = 3a^a_n(k-1)-3a^a_n(k-2) + a^a_n(k-3)$.
\end{corollary}
\begin{proof}
From Theorem \ref{thm:main2-even-acute}, we have $a^a_n(k) = 2a^a_n(k-1)-a^a_n(k-2) + \lfloor \frac{(n-1)^2}{2}\rfloor$
and $a^a_n(k-1) = 2a^a_n(k-2)-a^a_n(k-3) + \lfloor \frac{(n-1)^2}{2}\rfloor$.  Subtracting these two equations and combining terms we reach the conclusion.
\end{proof}

A similar homogeneous linear recurrence for right isosceles triangles is:

\begin{theorem} \label{thm:main2-even-right}
Suppose  $n >2$ and $ k > \max(5,(n-1)^2)$.  Then
$a^r_n(k) = 2a^r_n(k-1)-a^r_n(k-2)$.
\end{theorem}

Similarly, we have

\begin{theorem} \label{thm:main3}
Suppose that $n$ is even and $ k = (n-1)^2+1$. Then
$a_n(k) = 2a_n(k-1)-a_n(k-2) + \left\lfloor \frac{(n-1)^2}{2}\right\rfloor + 4$.  
\end{theorem}

\begin{theorem} \label{thm:main3-obtuse}
Suppose that $n>2$ is even and $k = (n-1)^2+1$.  Then $a^o_n(k) = 2a^o_n(k-1)-a^o_n(k-2)+4 $ .  
\end{theorem}

Because of the recurrence relation in Theorem \ref{thm:main}, the generating functions for $a_n(k)$ for $k\geq 1$ will be of the 
form $\frac{p_n(x)}{(x - 1)^{3}(x + 1)}$.  The polynomials $p_n(x)$ for the first few values of $n$ are shown in Table \ref{tbl:generate}.
\begin{table}[htbp]
\begin{center}
\begin{tabular}{|l|l|}
\hline
$n$ & $p_n(x)$ \\ \hline\hline
$2$ & $2 x (2 x^{2} - x - 2)$\\\hline
$3$ & $2 x (2 x^{4} + 4 x^{3} + 2 x^{2} - 8 x - 5)$\\\hline
$4$ & $4 x (x^{10} - x^{8} + 2 x^{6} + x^{5} + 4 x^{4} + 4 x^{3} - 3 x^{2} - 9 x - 4)$\\\hline
$5$ & $4 x (x^{16} - x^{14} + 2 x^{10} + 2 x^{9} - x^{8} - x^{7} + 5 x^{6} + $\\
& $6 x^{5} + 6 x^{4} + x^{3} - 8 x^{2} - 15 x - 6)$\\\hline
$6$ & $2 x (2 x^{26} - 2 x^{24} + 2 x^{22} - 2 x^{20} + 6 x^{16} - 4 x^{14} + 2 x^{13} - 2 x^{11} + $\\
& $6 x^{10} + 12 x^{9} - 6 x^{7} + 26 x^{6} + 24 x^{5} + 6 x^{4} - 8 x^{3} - 30 x^{2} - 43 x - 16)$ \\\hline
$7$ & $2 x (2 x^{36} - 2 x^{34} + 2 x^{28} + 2 x^{26} - 4 x^{24} + 4 x^{22} - 4 x^{20} + 4 x^{19} - 2 x^{17} +$\\
& $ 10 x^{16} -6 x^{14} + 4 x^{13} + 2 x^{12} - 2 x^{11} + 10 x^{10} + 18 x^{9} + 2 x^{8} + 4 x^{7} + $\\
&$40 x^{6} + 22 x^{5} - 2 x^{4} - 20 x^{3} - 44 x^{2} - 58 x - 21)$\\\hline
$8$ & $4 x (x^{50} - x^{48} + x^{46} - x^{44} + 3 x^{36} - 2 x^{34} - x^{32} + 2 x^{28} + 2 x^{26} + x^{25} - 3 x^{24} - x^{23}$\\
&$ + 3 x^{22} + x^{21} - 4 x^{20} + 4 x^{19} + x^{18} - 3 x^{17} + 7 x^{16} - 3 x^{14} + 5 x^{13} + 3 x^{12} - 2 x^{11} $\\
&$+ 6 x^{10} + 13 x^{9} + 10 x^{8} + 7 x^{7} + 19 x^{6} + 9 x^{5} - 7 x^{4} - 17 x^{3} - 29 x^{2} - 37 x - 13) $\\\hline
\end{tabular}
\caption{Numerator of generating functions of $a_n(k)$.}\label{tbl:generate}
\end{center}
\end{table}

Similarly, the polynomials for obtuse, acute and right isosceles triangles are shown in Tables \ref{tbl:generate-obtuse}-\ref{tbl:generate-right}.
We see that the denominator for the generating function of $a^a_n(k)$ can be reduced to $(x-1)^3$ corresponding to the recurrence relation in Corollary \ref{cor:main2-even-acute}.
Similarly, we see that the denominator for the generating function of $a^r_n(k)$ can be reduced to $(x-1)^2$ corresponding to the recurrence relation in
Theorem \ref{thm:main2-rectangular}.

\begin{table}[htbp]
\begin{center}
\begin{tabular}{|l|l|}
\hline
$n$ & $p^o_n(x)$ \\ \hline\hline
$2$ & $-2 x^4$\\\hline
$3$ & $- 2 x^{4} \left(x^{2} + 2\right)$\\\hline
$4$ & $2 x^{3} \left(2 x^{8} - 2 x^{6} - x^{5} - 2 x^{3} + 2 x^{2} - 3 x - 2\right)$\\\hline
$5$ & $2 x \left(2 x^{16} - 2 x^{14} + 4 x^{10} + x^{9} - 4 x^{8} - 4 x^{7} - x^{5} + 4 x^{4} - 6 x^{3} - 3 x^{2} - 1\right)$\\\hline
$6$ & $2 x (2 x^{26} - 2 x^{24} + 2 x^{22} - 2 x^{20} + 6 x^{16} - 6 x^{14} + 2 x^{13} - 3 x^{11} + 6 x^{10}$\\
&$ + 2 x^{9} - 6 x^{8} - 7 x^{7} + 4 x^{6} + 2 x^{4} - 9 x^{3} - 4 x^{2} - 2)$\\\hline
$7$ & $2 x (2 x^{36} - 2 x^{34} + 2 x^{28} + 2 x^{26} - 4 x^{24} + 4 x^{22} - 4 x^{20} + 2 x^{19} + 10 x^{16} - 2 x^{15} - 10 x^{14}$\\
&$ + 5 x^{13} - 8 x^{11} + 8 x^{10} + 3 x^{9} - 8 x^{8} - 6 x^{7} + 8 x^{6} - 3 x^{5} - 11 x^{3} - 4 x^{2} - x - 4)$\\\hline
$8$ & $2 x (2 x^{50} - 2 x^{48} + 2 x^{46} - 2 x^{44} + 6 x^{36} - 4 x^{34} - 2 x^{32} + 4 x^{28} + 2 x^{26} + 2 x^{25} - $\\
&$ 6 x^{24} - 2 x^{23} + 8 x^{22} + 2 x^{21} - 8 x^{20} + 2 x^{19} + 16 x^{16} - 5 x^{15} - 16 x^{14} + 10 x^{13} -  $\\
&$15 x^{11} + 10 x^{10} +4 x^{9} - 4 x^{8} - 5 x^{7} + 6 x^{6} - 6 x^{5} - 2 x^{4} - 13 x^{3} - 4 x^{2} - 2 x - 6)$\\\hline
\end{tabular}
\caption{Numerator of generating functions $\frac{p^o_n(x)}{(x - 1)^{3}(x + 1)}$ of $a^o_n(k)$.}\label{tbl:generate-obtuse}
\end{center}
\end{table}

\begin{table}[htbp]
\begin{center}
\begin{tabular}{|l|l|}
\hline
$n$ & $p^a_n(x)$ \\ \hline\hline
$2$ & $0$\\\hline
$3$ & $2 x^{2} \left(x + 1\right) \left(3 x - 4\right)$\\\hline
$4$ & $2 x^{2} \left(x + 1\right) \left(2 x^{4} - x^{3} + 3 x^{2} + 3 x - 9\right)$\\\hline
$5$ & $2 x^{2} \left(x + 1\right) \left(2 x^{7} - 2 x^{6} + x^{5} + 5 x^{4} + 3 x^{3} - x^{2} + 3 x - 15\right)$\\\hline
$6$ & $2 x^{2} \left(x + 1\right) (2 x^{12} - 2 x^{11} + 2 x^{10} - 2 x^{9} + 7 x^{7} - 5 x^{6} + x^{5} $\\
&$+ 15 x^{4} + 2 x^{3} - 6 x^{2} + 2 x - 22)$\\\hline
$7$ & $2 x^{2} \left(x + 1\right) (2 x^{17} - 2 x^{16} + 2 x^{13} + 2 x^{12} - 4 x^{11} + 4 x^{10} - x^{9}$\\
&$ - x^{8} + 11 x^{7} - 7 x^{6} + 10 x^{5} + 14 x^{4} + 2 x^{3} - 13 x^{2} + 2 x - 30)$\\\hline
$8$ & $2 x^{2} \left(x + 1\right) (2 x^{24} - 2 x^{23} + 2 x^{22} - 2 x^{21} + 6 x^{17} - 4 x^{16} - 2 x^{15} +$\\
&$  4 x^{13} + 4 x^{12} - 7 x^{11} + 9 x^{10} - 3 x^{9} - x^{8} + 16 x^{7} + 10 x^{5} + 12 x^{4} + x^{3} - $\\
&$21 x^{2} + 3 x - 39)$\\\hline
\end{tabular}
\caption{Numerator of generating functions $\frac{p^a_n(x)}{(x - 1)^{3}(x + 1)}$ of $a^a_n(k)$.}\label{tbl:generate-acute}
\end{center}
\end{table}

\begin{table}[htbp]
\begin{center}
\begin{tabular}{|l|l|}
\hline
$n$ & $p^r_n(x)$ \\ \hline\hline
$2$ & $2 x \left(x - 1\right) \left(x + 1\right) \left(x + 2\right)$\\\hline
$3$ & $2 x \left(x - 1\right) \left(x + 1\right) \left(x^{3} + 2 x^{2} + 4 x + 5\right)$\\\hline
$4$ & $2 x \left(x - 1\right) \left(x + 1\right) \left(x^{5} + 2 x^{4} + 4 x^{3} + 6 x^{2} + 9 x + 8\right)$\\\hline
$5$ & $2 x \left(x - 1\right) \left(x + 1\right) \left(x^{7} + 2 x^{6} + 4 x^{5} + 6 x^{4} + 9 x^{3} + 12 x^{2} + 15 x + 11\right)$\\\hline
$6$ & $2 x \left(x - 1\right) \left(x + 1\right) (x^{9} + 2 x^{8} + 4 x^{7} + 6 x^{6} + 9 x^{5} + 12 x^{4} + 16 x^{3} + 20 x^{2}$\\
&$ + 21 x + 14)$\\\hline
$7$ & $2 x \left(x - 1\right) \left(x + 1\right) (x^{11} + 2 x^{10} + 4 x^{9} + 6 x^{8} + 9 x^{7} + 12 x^{6} + 16 x^{5} + 20 x^{4}$\\
&$ + 25 x^{3} + 29 x^{2} + 27 x + 17)$\\\hline
$8$ & $2 x \left(x - 1\right) \left(x + 1\right) (x^{13} + 2 x^{12} + 4 x^{11} + 6 x^{10} + 9 x^{9} + 12 x^{8} + 16 x^{7} + 20 x^{6}$\\
&$ + 25 x^{5} + 30 x^{4} + 36 x^{3} + 38 x^{2} + 33 x + 20)$\\\hline
\end{tabular}
\caption{Numerator of generating functions $\frac{p^r_n(x)}{(x - 1)^{3}(x + 1)}$ of $a^r_n(k)$.}\label{tbl:generate-right}
\end{center}
\end{table}

\section{Optimal value for $K(n)$}
Theorem \ref{thm:main3} shows that $(n-1)^2+3$ is the best value for $K(n)$ when $n$ is even.  Next we show that
$(n-1)^2+2$ is the best value for $K(n)$ when $n$ is odd.

Consider the case where $n>2$ is odd and $k = (n-1)^2$.  Then $k$ is even and by setting $y = \frac{k}{2}$, $w = \frac{k}{2}-1$, $x = 1$, $u = n-1$, we get $x^2+y^2 = u^2+w^2$ where $x\neq u$ and $y\neq w$ and $y+w = k-1$.
This corresponds to $4$ additional isosceles triangles in $B(k)$.  These triangles are right triangles for $n=3$ (Fig. \ref{fig:three}) and
obtuse for $n > 3$.

\begin{figure}[htbp]
\centerline{\includegraphics{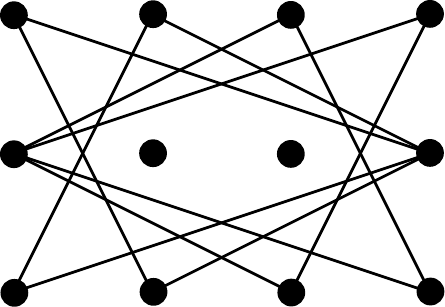}}
\caption{$4$ right isosceles triangles in $B(k)$ corresponding to $x=1$ for the case $k=4$, $n=3$.}\label{fig:three}
\end{figure}

\begin{theorem} \label{thm:main4}
Suppose that $n$ is odd and $ k = (n-1)^2$. Then
$a_n(k) = 2a_n(k-1)-a_n(k-2)  + \left\lfloor \frac{(n-1)^2}{2}\right\rfloor + 4$.  
\end{theorem}

\begin{theorem} \label{thm:main4-obtuse}
Suppose that $n > 3$ is odd and $k = (n-1)^2$. Then $a^o_n(k) = 2a^o_n(k-1)-a^o_n(k-2)+4 $ .  
\end{theorem}

This means that $K(n) = (n-1)^2+3$ for $n$ is even and $K(n) = (n-1)^2+2$ for $n$ is odd are the best possible values for $K(n)$ in the conjecture in Section \ref{sec:intro} as expressed in Theorems \ref{thm:main} and \ref{thm:main-even}.  They are also optimal when restricted to the set of obtuse isosceles and $n>3$.
In addition, 
$a_n(k) = 2a_n(k-1)-2a_n(k-3)+a_n(k-4)-4$ if $k = (n-1)^2+3$ and $n$ is even or if $k=(n-1)^2+2$ and $n$ is odd.  This is due to the fact that for these values of $n$ and $k$, $c_n(k-2) = c_n(k)+4$.

\section{Subsets of points on a regular grid containing the vertices of isosceles triangles}
Consider an $n$ by $k$ regular grid.  Let $S(n,k) = r$ be defined as the smallest number $r$ such that any $r$ points in the grid contains the vertices of an isosceles triangle of nonzero area.  This is equivalent to finding the constellation with the largest number of points $T(n,k)$ such that no three points form an isosceles triangle as $S(n,k) = T(n,k)+1$.  The values of $T(n,k)$ for small $n$ and $k$ are listed in OEIS \href{http://oeis.org/A271914}{A271914} where it was conjectured that $T(n,k) \leq n+k -1$. The next Lemma shows that this conjecture is true for $k=1,2$.

\begin{lemma} \label{lem:tnk}
$T(n,k)$ satisfies the following properties:
\begin{enumerate}
\item $T(n,k) = T(k,n)$. If $m\leq n$, then $T(m,k)\leq T(n,k)$.
\item $ T(n,k) \geq \max(n,k)$, 
\item $T(n,1) = n$,
\item (subadditivity) $ T(n,k+m) \leq T(n,k) + T(n,m)$,
\item $ T(n,2) = n $ for $ n > 3$,
\item Suppose $n > 4$. Then $T(n,3) \geq n+1$ if $n$ is odd and $T(n,3) \geq n  + 2$ if $n$ is even.
\end{enumerate}
\end{lemma}
\begin{proof}
Property 1 is trivial. Properties 2 and 3 are clear, by considering points all in one row (or column) only. To show property $4$, suppose $T(n,k) + T(n,m) + 1$ points are chosen in an $n$ by $k+m$ grid, and consider a decomposition into an $n$ by $k$ grid and an $n$ by $m$ grid. Either the $n$ by $k$ grid has more than $ T(n,k)$ points or the $n$ by 
$m$ grid has more than $T(n,m)$ points.  In either case, there is an isosceles triangle.  
 Property $4$ implies that $T(n+2,2) \leq T(n,2) + T(2,2) = T(n,2) + 2$ and property $2$ implies $T(n, 2) \geq n$ for $ n \geq 2$. Thus property $5$ follows from properties $2$ and $4$.
Suppose $n > 4$. For $n$ odd, the points $\{(1,i): 2\leq i\leq n\}$ along with the points $(2,1)$ and $(3,1)$ shows that $T(n,3)\geq n+1$ (Fig. \ref{fig:odd3}).
If $n$ is even, the points   $\{(1,i): 2\leq i\leq n-1\}$ along with the points $(2,1)$, $(3,1)$, $(2,n)$ and $(3,n)$ shows that $T(n,3) \geq n+2$ (Fig. \ref{fig:even3}).
\end{proof}

\begin{figure}[htbp]
\centerline{\includegraphics{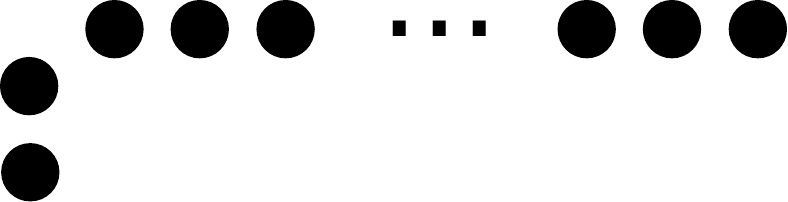}}
\caption{A constellation of $n+1$ points for which there are no $3$ points forming an isosceles triangle for the case when $n > 4$ is odd.}\label{fig:odd3}
\end{figure}

\begin{figure}[htbp]
\centerline{\includegraphics{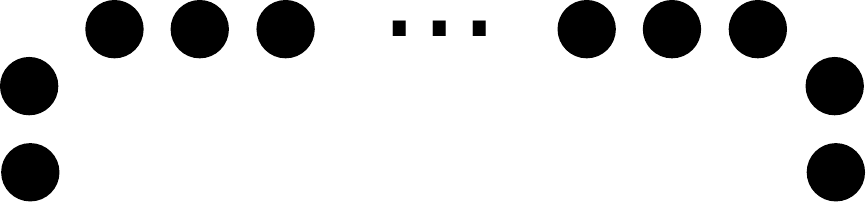}}
\caption{A constellation of $n+2$ points for which there are no $3$ points forming an isosceles triangle for the case when $n > 4$ is even.}\label{fig:even3}
\end{figure}

Note that for $n$ odd, the constellation in Fig. \ref{fig:even3} for $n-1$ which has $n-1+2 = n+1$ points will also show that $T(n,3) \geq n+1$. We conjecture the following:
\begin{conjecture} \label{conj:one}
If $n$ is even, then $T(n,k) \leq n+k-2$ for $k \geq 2n$.
\end{conjecture}

\begin{conjecture}
For $n>4$, $T(n,3) = n+1$ if $n$ is odd and $T(n,3) = n+2$ if $n$ is even. 
\end{conjecture}

Note that by Lemma \ref{lem:tnk} Conjecture \ref{conj:one} is true for $n = 2$.

\end{document}